\documentclass[smallextended,referee,envcountsect]{gOPT2e}
\usepackage{hyperref}  
\hypersetup{
    colorlinks=true,                         
    linkcolor=blue, 
    citecolor=red, 
    urlcolor=blue  } 

\usepackage{amsfonts, color}
\usepackage{amssymb, graphicx}
\usepackage[applemac]{inputenc}
\usepackage{subfigure}

\theoremstyle{plain}
\newtheorem{theorem}{theorem}[section]

\newtheorem{proposition}[theorem]{Proposition}

\theoremstyle{remark}
\newtheorem{remark}{Remark}

\theoremstyle{definition}
\newtheorem{definition}{Definition}
\def\disp{\displaystyle}

\def\tto{\;{\lower 1pt \hbox{$\rightarrow$}}\kern -10pt
\hbox{\raise 2pt \hbox{$\rightarrow$}}\;}

\def\lam{\lam}

\def\ve{\varepsilon}

\def\ox{\bar{x}}

\def\cl*co{\mbox{\rm cl}^*\mbox{\rm co}\,}
\def\cl{\mbox{\rm cl}\,}

\def\cl{\mbox{\rm cl}\,}

\def\st{\stackrel}

\def\hs7{\hspace*{7pt}}

\renewcommand{\theequation}{\thesection.\arabic{equation}}
\definecolor{Orange} {cmyk}{0,0.61,0.87,0}

\begin{document}

\title{{\itshape On Extended Versions of Dancs-Heged\"us-Medvegyev's Fixed Point Theorem}}

 \author{Truong Q. Bao$^{\rm a}$and Michel A. Th\'era$^{\rm b}$$^{\ast}$\thanks{$^\ast$Corresponding
author. Email: michel.thera@unilim.fr \vspace{6pt}}\\$^{a}${\em{Department of Mathematics $\&$ Computer Science, Northern Michigan University, Marquette, Michigan 49855, USA}}; $^{b}${\em{University of Limoges, France $\&$ Centre for Informatics and Applied Optimisation, Federation University, Australia}}\\\received{august 2015} }

\maketitle

\begin{abstract}
In this article we  establish some fixed point (known also as critical point, invariant point)  theorems in \textit{quasi-metric spaces}. Our results unify and further extend in some regards the fixed point theorem proposed by Dancs et al. (1983), the results given by Khanh and Quy (2010, 2011), the preorder  principles established by Qiu (2014), and the results obtained by Bao et al. (2015). In addition, we provide examples to illustrate that the improvements of our results are significant.

\begin{classcode}49J53,  49J52 47J30,  54H25, 90C29,  90C30. \end{classcode}

\end{abstract}

\section{Introduction}
 
The celebrated Ekeland variational principle 
has been recognized as a fundamental tool in the study of various aspects 
of optimization theory  and variational  analysis.  Since it has been established, it  has found    many applications in different fields in Analysis.  For instance,  it has  been used to prove the infinite-dimensional mountain path theorem of Ambrosetti and Rabinowitz  \cite{AR73} and has been the key ingredient for proving new variational principles such as the Borwein-Preiss variational principle \cite{BP87}.  
It has provided
simple and elegant proofs of known results
such as 
the  
Caristi fixed point theorem in complete metric spaces \cite{c76} (in fact  the two results are equivalent).  It is  well established  that Dancs-Heged\"us-Medvegyev's fixed point theorem \cite[Theorem~3.1]{dhm83} has served 
 as a significant tool in proving  Ekeland's variational principle \cite{e79} and its extensions to  vector and set optimization; the reader is referred  for instance to \cite{bms15SVVA,bms15b, bks15, h05, kq10, kq11, q14a,q14b}. It is important to emphasize that  the Dancs-Heged\"us-Medvegyev  fixed point theorem is equivalent to  Ekeland's variational principle \cite{e79} in the sense that one implies the other. 
 The plan of the paper is organized as follows. We begin in section 2 with recalling the  Dancs-Heged\"us-Medvegyev fixed point theorem  and some of its recent  developments. Through this section we recall also some concepts and notations that we will use in the rest of the paper. Our work  requires the concept  of quasi-metric space, which we review in section 3. Armed with the  previous results and such quasi-metric tools, in section 4,  we  establish in Theorem 4.1 an  unified version of  Dancs-Heged\"us-Medvegyev fixed point theorem, as well as an  ``all sequences"  version of Theorem 4.1 .Finally, remarks on further research topics are given in section 5.

 \section{ Some  recent developments: a brief survey}
 For convenience of the reader, let us recall   the Dancs-Heged\"us-Medvegyev fixed point theorem  and some of its recent  developments.  Throughout,  we will use the notation``$\Phi: X\rightrightarrows X$" to denote  a  set-valued mapping,  that is a mapping  assigning to each point $x\in X$, a subset (possibly empty) $\Phi(x)$ of $X$ and  we say that 
  $\{x_n\} \subset X$ is a generalized Picard sequence of $\Phi$, 
 if $x_{n+1} \in \Phi(x_n)$ for all $n\in \mathbb{N}$. 

\begin{theorem} {\bf (\cite[Theorem~3.1]{dhm83}).}\label{thm-DHM} Let $(X,d)$ be a complete metric space, and let $\Phi: X\tto X$ be a set-valued mapping  satisfying the following conditions:
\begin{itemize}
\item[\bf (A1)] $\Phi(x)$ is a closed set for all $x\in X$;
\item[\bf (A2)] $x\in \Phi(x)$ for all $x\in X$;
\item[\bf (A3)] $x_2\in\Phi(x_1)\Longrightarrow\Phi(x_2)\subset\Phi(x_1)$ for all $x_1,x_2\in X$;
\item[\bf (A4)] For each  generalized Picard sequence  $\{x_n\}\subset X$ of $\Phi$, $\disp\lim_{n\to +\infty}d(x_n, x_{n+1}) = 0$.
\end{itemize}
Then, for every starting point $x_0\in X$,  there is a convergent 
sequence $\{x_n\} \subset X$ whose limit $x_{\ast}$ is a fixed point of $\Phi$, i.e., $\Phi(x_{\ast})=\{x_{\ast}\}$.
\end{theorem}

In \cite{kq10, kq11}, Khanh and Quy presented an extension of Theorem~\ref{thm-DHM} in order to establish a new version of Ekeland's variational principle for  weak $\tau$-functions. 

\begin{definition} {\bf($\tau$-functions and weak $\tau$-functions \cite{kq10,kq11}).} \label{def:tau}
Let $(X,d)$ be a metric space. A bifunction $p : X\times X\to \mathbb R_+$ is called a \textit{$\tau$-function}  whenever  the following four conditions hold:
\begin{description}
\item[\rm($\tau{1}$)] $p(x, z) \leq p(x, y) + p(y, z)$  (triangle inequality);
\item[\rm($\tau{2}$)]  for all $x\in X,\; p(x,\cdot)$ is lower semicontinuous (lower semicontinuity);
\item[\rm($\tau{3}$)] for all  sequences  $\{x_n\}$, $\{y_n\}$ with $\disp\lim_{n\to \infty} p(x_n, y_n) = 0$ and $\disp\lim_{n\to\infty} \sup_{m > n}p(x_n, x_m)  = 0$,  one has $\disp\lim_{n\to\infty} d(x_n, y_n) = 0$ ($p$-convergence implies $d$-convergence); 
\item[\rm($\tau{4}$)]  $p(x, y) = 0$ and $p(x, z) = 0$ imply that $y = z$ (indistancy implies coincidence).
\end{description}
A bifunction $p : X\times X\to \mathbb R_+$ is called   \textit{weak $\tau$-function} whenever  it satisfies conditions {\rm($\tau{1}$)}, {\rm($\tau{3}$)}, and {\rm($\tau{4}$)}.
\end{definition}

\begin{theorem}{\bf (\cite[Lemma~3.4]{kq11})} \label{DHM-kq} 
Let $(X,d)$ be a metric space,  $p$ be a weak $\tau$-function on $X$, and  $\Phi: X\tto X$ be  a set-valued mapping.   Suppose given a generalized Picard sequence $\{x_n\}\subset X$   of $\Phi$ convergent to $\ox$  with respect to $p$ in the sense that $\disp\lim_{n\to\infty} p(x_n,\ox) = 0$ with the following properties:
\begin{itemize}
\item[\bf (B1)]   $\Phi(x_{n+1})\subset \Phi(x_n)$ for all $n\in \mathbb{N}$;
\item[\bf (B2)] $\disp\lim_{n\to \infty} \sup_{u\in \Phi(x_n)}\;p(x_n,u) = 0$;

\item[\bf (B3)] $\ox \in \Phi(x_n)$ for all $n\in \mathbb{N}$.
\end{itemize}
Then, 
$$
\bigcap_{n\in \mathbb{N}} \Phi(x_n) = \{\ox\}.
$$
Assume, in addition, that
\begin{itemize}
\item[\bf (B4)] $\Phi(\ox) \neq \emptyset$ and $\Phi(\ox)\subset \Phi(x_n)$ for all $n\in \mathbb{N}$.
\end{itemize} 
\noindent Then $\ox$ is an invariant point of $\Phi$, i.e., $\Phi(\ox) = \{\ox\}$.
\end{theorem}

\begin{remark} In \cite[Lemma~3.4]{kq11}, Khanh and Quy imposed  assumptions on one generalized Picard sequence under consideration instead of all sequences in the original or similar results. It is important to emphasize that Dancs et al.'s proof in \cite[Theorem~3.1]{dhm83} also holds under the validity of conditions {\rm(B1)}--{\rm(B4)}.  In fact, conditions {\rm(A1)}--{\rm(A4)} ensure the existence of a generalized Picard sequence of $\Phi$ which satisfies condition {\rm(B1)}--{\rm(B4)}. It is worth emphasizing that  the proof of \cite[Lemma~3.4]{kq11} is nothing but the middle part of Khanh-Quy's version of Ekeland's variational principle.
\vspace*{.05in}
\end{remark}

In \cite{q14a}, Qiu established a general preorder  principle from which 
most of the known set-valued Ekeland variational principles and their improvements were derived. However, it could
not imply Khanh and Quy's EVP in the afore-mentioned  papers \cite{kq11}   in which
a weak $\tau$-function plays the role of the metric in the original principle since the generalized distance between two distinct points $x$ and $y$ of a weak $\tau$-function $p(x,y)$ may be zero. Then, Qiu further revised it to a more general version in \cite[Theorem~2.1]{q14b}.

\begin{definition} {\bf (preordered and ordered sets)}
Let $\Xi$ be a nonempty set and $Q\subset \Xi\times \Xi$ be a subset of the cartesian  product  $ \Xi\times \Xi$. Let us  define a binary relation $\preceq$ associated to $Q$ on $\Xi$ by
$$
v \preceq z \;:\Longleftrightarrow\; (v,z) \in Q.
$$
The binary relation $\preceq$ is a \textit{ preorder}; known also as a {\em quasiorder}, whenever it satisfies the following properties:
\begin{itemize}
\item[] $\big[\;\forall \; z\in \Xi, \; z \preceq z\; \big]$  (reflexivity) and  
\item[] $\big[\; \forall\; z, z', z''\in \Xi,\; z \preceq z' \;\wedge\; z' \preceq z'' \;\Longrightarrow\; z \preceq z''\; \big]$ (transitivity).
\end{itemize}
A set   equipped with a preorder is called a \textit{ preordered }set. When $\Xi = Z$ is a vector space, we call the pair  $(Z,\preceq)$ a preordered vector space. 
If a preorder is also {\em antisymmetric}, i.e., 
\begin{itemize}
\item[---] $z \preceq v\;\wedge\;v \preceq z\;\Longrightarrow\;v = z$ (antisymmetry),
\end{itemize}  
then it is a\textit{ partial order}. 
\end{definition}

\begin{theorem}{\bf (\cite[Theorem~2.1]{q14b})} \label{DHM-qiu}
 Let $(X,\preceq)$ be a preordered  set and consider the level-set mapping $S: X\tto X$ of the set $X$ with respect to the preorder $\preceq$ defined by
\begin{eqnarray}\label{def-S}
S(x) := \{u\in X|\; u\preceq x\}.
\end{eqnarray}
Let $x_0\in X$  be such that $S(x_0)\neq \emptyset$ and consider $\varphi: (X,\preceq) \to \mathbb{R}\cup\{\pm\infty\}$   an extended real-valued function which is monotone with respect to $\preceq$. Suppose that
\begin{itemize}
\item[\bf(C1)] $-\infty < \inf\{ \varphi (x)|\; x\in S(x_0)\} < +\infty$;
\item[\bf(C2)] For any $x\in S(x_0)$ with $-\infty < \varphi (x) < +\infty$ and for any $z_1, z_2 \in S(x)$ with $z_1\neq z_2$, one has $\varphi(x) > \min\{\varphi (z_1), \varphi (z_2)\}$;
\item[\bf(C3)] For any generalized Picard sequence $\{x_n\}\subset S(x_0)$ of $S$  satisfying
$$
\varphi (x_n) - \inf_{x\in S(x_{n\!-\!1})}\;\varphi (x) \to 0 \; \mbox{ as } \; n\to \infty,
$$
there exists $x_\ast\in X$ such that $x_\ast\in S(x_n)$ for all $n\in \mathbb{N}$.
\end{itemize}
Then, there exists $\ox\in X$ such that 
\begin{itemize}
\item[\bf(i)] $\ox\in S(x_0)$;
\item[\bf(ii)] $S(\ox) \subset \{\ox\}$ which holds as equality provided that $S(\ox)\neq \emptyset$.  
\end{itemize}
\end{theorem}

\begin{remark} Since $\preceq$ is a preorder, the level-set mapping $S$ defined by \eqref{def-S} automatically satisfies conditions {\rm(A2)} and {\rm(A3)}. It plays the role  of the set-valued mapping $\Phi$ in Theorems~\ref{thm-DHM} and \ref{DHM-kq}. The function $\varphi$ is nothing but   a  utility function associated to   the preorder $\preceq$. It allows the author not to impose some topological properties on $X$. Later, we will show that some relaxation of the completeness and  separation  properties of the metric space $(X,d)$ in Dansc et al.'s result are fulfilled under the imposed assumptions {\rm(C1)}--{\rm(C3)}.
\end{remark}

In \cite{bms15SVVA, bms15b}, Bao, Mordukhovich and Soubeyran established a far going extension of Dancs-Heged\"us-Medvegyev's fixed point theorem for parametric  multifunctions  in quasi-metric spaces. This extension  can also be interpreted as an existence theorem of minimal points with respect to reflexive and transitive preferences for sets in products spaces with  possible applications to  behavioral sciences. Below is a simple version of \cite[Theorem~2.3]{bms15SVVA}.

 
\section{Quasi-metric spaces: definition, basic properties and examples}
\begin{definition} {\bf (quasi-metric spaces)}\label{quasi}
A quasi-metric space  (also called quasi-pseudo-metric space by Reilly et al. \cite{rsv82}) is a pair  $(X,q)$ consisting of a set $X$  and  a function $q:X\times X\longmapsto\mathbb R_+:=[0,\infty)$ on $X\times X$ having the following three properties:
\begin{itemize}
\item[\bf (i)] $q(x,x^{\prime})\ge 0$ for all $x,x^\prime\in X$ and $q(x,x)=0$ for all $x\in X$ (positivity);
\item[\bf (ii)] $q(x,x'')\le q(x,x^{\prime})+q(x^{\prime},x'')$ for all $x, x^{\prime},x''\in X$ (triangle inequality).
\end{itemize}
\end{definition}

\noindent 

There is an abundant literature devoted to ``distances"  where the requirement of symmetry is omitted.  Quasi-metrics are common in real life. For example, given a set $X$ of mountain villages, the typical walking times between elements of $X $ form a quasi-metric because traveling up hill takes longer than traveling down hill. Another example is a  geometry topology having one-way streets, where a path from point $A $  to point $B$  comprises a different set of streets than a path from $B$ to $A$. These ``metrics"  have some interest in topology,  but they are also used in applied mathematics in the calculus of variation. Recently, Bao et al. studied in \cite{bms15JOTA, bms15a, bks15} some mathematical models arising in some
areas of behavioral sciences (called sometimes
``theories of stability/stay and change"). It seems that everyone agrees that the cost to change in these models does not satisfy the symmetry property.  
 Note that 
  several terminologies are used for the   concept of what we call in this paper, quasi-metric: Mennuci \cite{menucci} uses the term \textit{asymmetric semidistance},  Cobza\c{s} \cite{c13}  and Reilly et al. \cite{rsv82} speak about \textit{quasi-pseudo-metrics}, while Deza et al. \cite{deza},  use  the name \textit{quasi semi-metric}  and Mainik and Mielke \cite{mainik05} employ the name of dissipation distance.
 
 As well-known, if in addition, a quasi-metric  satisfies the \textit{symmetry}  property $q(x,x^{\prime })=q(x^{\prime },x)$ for all $x,x^\prime\in X$, then $q$ is a \textit{metric}.
Part (ii) in the previous definition of a  quasi-metric was formalized
 by Hausdorff   in the  celebrated monography  ``\textit{Grundz\"{u}ge der Mengenlehre}" \cite[p. 145–146]{haus}
which is  considered as the foundation of  the theory of topological and metric spaces (see details in \cite{deza} and \cite{wilson}).  Part (ii) was  first formalized
  by Fréchet \cite{frechet}   and later treated by  Hausdorff \cite{haus}. 
 Some examples of quasi-metrics   are listed below: 
\begin{itemize}
\item
  the Sorgenfrey quasi-metric  on  $\mathbb R$,   defined by  $q(x,y)= y-x$ if $y\geq x$ and $q(x,y)=1 $ otherwise.  This quasi-metric describes the process of filing down a metal stick: it is easy to reduce its size, but it is difficult or impossible to grow it;
\item 
the quasi-metric on $\mathbb R$ defined by  $q(x,y)=\max (y-x,0)$;
\item
 the real half-line quasi-semi-metric   defined by $ q(x,y) = \max(0, \ell n \frac{y}{x})$ on the set of strictly positive reals;
\item 
 the circular-railroad distance,  see, \cite[Example 2.2]{smith}.   Imagine a circular railroad line 
which moves only in a counterclockwise direction around a circular track,
represented by the unit circle  $\mathcal{S}^1$. The  circular-railroad quasi-metric  from any point, $x\in \mathcal{S}^1$, to any other point,
$y\in \mathcal{S}^1$, is simply the counterclockwise circular arc length from $x$ to $y$ in $\mathcal{S}^1$;
\item  the dissipation distance related to the energetic formulation of energetic models
for rate-independent systems \cite{mainik05}:  consider $X:=\;\{u\in L^1(\Omega, \mathbb R^p):\;\Vert u\Vert_\infty \leq 1\}$ equipped with the weak $L^1$- topology and the dissipation distance defined by $q(u_1,u_2) = \Vert u_1-u_2\Vert_{L^1}.$
\item
the Minkowski gauge function  defined on $\mathbb R^n$  by   $q_B(x,y)=\inf\{\alpha>0:y-x\in \alpha B\}$, where $B$ is a convex  compact subset of $\mathbb R^n$.
\end{itemize}
 Since the conjugate bifunction $\overline{q} : X \times X \to \mathbb R_+$ of a quasi-metric defined by $\overline{q}(x,y) = q(y,x)$ is also a quasi-metric. Following Kelly \cite{kelly}, the space $(X,q,\overline{q})$ is called a bitopological spaces with two topologies:
\begin{itemize}
\item the topology $\tau_q$ generated by the balls with center $x\in X$ and radius $\varepsilon$ and defined by  $\mathbb{B}_q(x;r) := \{y\in X: q(x,y)<\varepsilon\}$;
\item  the topology $\tau_{\overline{q}}$ 
generated by the balls with center  $x\in X$ and radius  $\varepsilon$ and defined by 
$\mathbb{B}_{\overline{q}}(x;r) := \{y\in X:  \overline{q}(x,y) <\varepsilon\} = \{y\in X:  q(y,x) <\varepsilon\}$.
\end{itemize}
The balls $\mathbb{B}_q(x;r)$ and $\mathbb{B}_{\overline{q}}(x;r)$ are called forward and backward balls by Menucci \cite{menucci} and left and right balls by Cabza\c{s}  \cite{c13}.\vspace*{.05in}
 These two topologies allow us to define two notions of convergences associated to the quasi-metric $q$:
\begin{definition} {\bf (convergences in quasi-metric spaces).}
\begin{itemize}
\item[\bf (i)]
A sequence $ \{x_n\}$ is said to be backward convergent to $x_\infty$, if it is convergent with respect to the topology $\tau_q$, i.e.,  $\disp\lim_{n\to +\infty} q(x_\infty,x_n) = 0$.
\item[\bf (ii)]  A sequence $ \{x_n\}$ is said to be forward convergent to $x_\infty$, if  it is convergent with respect to the topology $\tau_{\overline{q}}$, i.e., $\lim_{n\to +\infty} q(x_n,x_\infty)=0$.
\end{itemize}
\end{definition}
 Since a quasi-metric   may fail to be symmetric, the quasi-distances $q(x_n,x_m)$ and $q(x_m,x_n)$ are different. The definition of Cauchy sequences in metric spaces takes two following forms. 
\begin{definition} {\bf (Cauchy sequences in quasi-metric spaces).}
\begin{itemize}
\item[\bf (i)]  A sequence $ \{x_n\}$ is said to be  forward Cauchy  if for every $\varepsilon >0$, there is some $N_\varepsilon \in \mathbb{N}$ such that  for every $n\geq N_\varepsilon$ and every $k\in \mathbb N$, then  $q(x_n, x_{n+k}) <\varepsilon$.
\item[\bf (ii)]  A sequence $ \{x_n\}$ is said to be backward Cauchy  if for every $\varepsilon >0$, there is some  $N_\varepsilon \in \mathbb N$ such that  for every $n\geq N_\varepsilon$ and and every $k\in \mathbb N$, then  $q( x_{n+k,} x_n) <\varepsilon$.
\end{itemize}
\end{definition}

Note that in a metric space the two concepts coincide with the usual concept of a  Cauchy-sequence.
\begin{remark}
At this point, it is important for the reader to   to be aware of the differences bettwen our definitions and the ones previously used. 
The backward convergence is termed either $q$-convergence or convergence w.r.t. $\tau_q$ in \citep{kelly, reilly}. The forward convergence is called as $\overline q$-convergence or convergence w.r.t. the topology $\tau_r$ in the aforementioned references, and left sequential convergence in \cite{bms15SVVA,bms15b, bks15, bms15JOTA} and many references therein. The forward Cauchy sequence is known as left-sequential Cauchy (also as lelf Cauchy or Cauchy) sequence in Bao et al. \cite{bms15SVVA,bms15b, bks15}, left-K-Cauchy in Reilly et al. \cite{reilly}.   The backward Cauchy notion is known as $p$-Cauchy in \cite[ Definition 2.10]{kelly}, right-K-Cauchy in Reilly et al. \cite{reilly}.
The forward completeness is used in \cite{bms15SVVA,bms15b, bks15, bms15JOTA} as left-sequential completeness and in Reilly \cite{reilly} as left-K-completness. The backward completeness was studied in Kelly \cite{kelly} under the name of  $p$-completness and in Reilly \cite{reilly} as right-K-completness.
\vskip 2mm
This change of notation is 
motivated by the following consideration (private communication with A. Soubeyran): denoting the state of an object at the time $n$ by $x_n$,  then the future (forward) state is $x_{n+1}$. Then, $q(x_n, x_{n+1})$ is the cost to change from $x_n$ to $x_{n+1}$. If the expected /ideal state is $x_\infty$, then the cost to change from the current state to the ideal state is$q(x_n, x_\infty)$ which should be called a forward cost.
\end{remark}
According to Reilly et al., \cite [Example 1. p. 130]{rsv82}, a  sequence could be forward convergent without being backward convergent and a sequence could be convergent without being forward or backward convergent.
 \begin{definition} {\bf (completeness in quasi-metric spaces).}
\begin{itemize}
\item[\bf (i)]  The space $(X,q)$ is   forward (resp. backward)  Hausdorff,  if  every forward  (resp. backward)  converging sequence has a unique forward (resp. backward) limit point.
\item[\bf (ii)] The space $(X,q)$ is  forward  (resp. backward) complete,  if every forward (resp. backward) Cauchy sequence is forward (resp. backward) convergent.
\item[\bf (iiii)]  The space $(X,q)$ is forward-backward (resp. backward-forward) complete,  if every forward (resp. backward) Cauchy sequence is backward (resp. forward)  convergent.
 \end{itemize}
\end{definition}
\begin{theorem} {\bf (\cite[Corollary~4.5]{bms15b})} \label{DHM-bms}
Let $(X,q)$ be a   forward complete and forward Hausdorff \footnote{\!$^{,2,3}$ the adjective `forward' (i.e., `left-sequential' in the original version) was obmitted for simplicity.
} quasi-metric space, and let $\Phi: X\tto X$ be a  set-valued mapping  satisfying the conditions:
\begin{itemize}
\item[\bf (D1)] $x\in\Phi(x)$ for all $x\in X$;
\item[\bf (D2)] $u\in\Phi(x)\Longrightarrow\Phi(u)\subset\Phi(x)$ for all $x,u\in X$;
\item[\bf (D3)] For each generalized Picard sequence $\{x_n\}$ of $\Phi$ 
being  forward convergent \footnotemark to $x_{\ast}$, then $x_{\ast}\in\Phi(x_n)$ for all $n\in\mathbb{N}$;
\item[\bf (D4)] For each generalized Picard sequence $\{x_n\}\subset X$, $\disp\lim_{n\to +\infty}q(x_n, x_{n+1})=0$.
\end{itemize}
Then, for every point $x_0\in X$ there is a generalized Picard sequence $\{x_n\}\subset X$ of $\Phi$ starting from $x_0$ and forward converging \footnotemark to an invariant point $\ox$ of $\Phi$, i.e., $\Phi(\ox)=\{\ox\}$.
\end{theorem}
In this paper, we establish an unified version for the afore-mentioned results. It takes the `one sequence' form of Theorem~\ref{DHM-kq} in the setting of Theorem~\ref{DHM-bms}.

\section{Main Results}

\begin{theorem}{\bf (a unified version of DHM's fixed point theorem).} \label{unified-version}
Let $(X,q)$ be a quasi-metric space, $\Phi: X\tto X$ be a  set-valued mapping, and  $\{x_k\}\subset X$ be a generalized Picard sequence of $\Phi$, i.e., $x_{n+1} \in \Phi(x_n)$ for all $n\in \mathbb{N}$. Assume that the following conditions hold:
\begin{itemize}
\item[\bf(E1)] $\Phi(x_{n+1})\subset\Phi(x_n)$ for all $n\in \mathbb{N}$;
\item[\bf(E2)] $\disp\lim_{n\to \infty}\disp\sup_{x\in\Phi(x_n)} q(x_n, x) = 0$;
\item[\bf(E3)] there is some $\ox\in X$ such that $\ox\in \Phi(x_n)$ for all $n\in \mathbb{N}$;
\item[\bf(E4)]
$\{x_n\}$ has at most one forward limit;
\end{itemize}
Then,   
\begin{eqnarray}\label{common-point}
\bigcap_{n\in \mathbb{N}} \Phi(x_n) = \{\overline{x}\}
\end{eqnarray}
where $\overline{x}$ is taken from {\rm(E3)}. Assume, in addition, that
\begin{itemize}
\item[\bf(E5)] $\Phi(\overline{x})\subset \Phi(x_n)$ for all $n\in \mathbb{N}$.
\end{itemize}
Then, $\ox$ is a nonvariant point of $\Phi$, i.e., $\Phi(\ox) \subset \{\ox\}$; it becomes an invariant point of $\Phi$ provided that $\Phi(\overline{x})\neq \emptyset$.
\end{theorem}
\begin{proof}
Suppose given a sequence $\{x_n\}\subset X$ satisfying conditions {\rm(E1)--(E4)}. Obviously, condition {\rm(E3)} says that
\begin{eqnarray}\label{nonempty}
\{\overline{x}\} \subset \disp\bigcap_{n\in \mathbb{N}} \Phi(x_n).
\end{eqnarray}
Next, we will prove that the intersection is a singleton. Assume, in addition to $\overline{x}$, that an element $\underline{x}$ also belongs to the left-hand side of  \eqref{nonempty}.  By condition {\rm(E2)}, we have
$\disp\lim_{n\to\infty} q(x_n, \overline{x}) 
= \disp\lim_{n\to\infty} q(x_n, \underline{x}) = 0$ which ensures that $\underline{x} = \overline{x}$ due to condition {\rm(E4)} and thus \eqref{nonempty} holds as  an equality, i.e., the common point condition \eqref{common-point} holds.
\vspace*{.05in}
Employing now condition {\rm(E5)} to \eqref{common-point} we obtain
$$
\Phi(\ox) \subset \disp\bigcap_{n\in \mathbb{N}} \Phi(x_n) = \{\ox\}.
$$ 
The proof is complete.
\end{proof}

\begin{proposition} \label{prop-Cauchy} The fulfilment of {\rm(E1)--(E2)} implies that the sequence $\{x_n\}$ is a forward Cauchy sequence with respect to the quasi-metric $q$ in $X$.
\end{proposition}
\begin{proof}
Condition {\rm(E2)} tells us that for every $\ve > 0$, there exists $N_\ve\in\mathbb{N}$ such that
$$
\sup_{u\in \Phi(x_n)} q(x_{n},u) < \ve\;\mbox{ whenever }\;n\ge N_\ve.
$$
Picking now any $n\ge m\ge N_\ve$, we have $x_m \in \Phi(x_m)\subset\Phi(x_n)$ due to {\rm(E1)} and thus 
$$q(x_m,x_n) \leq \sup_{u\in \Phi(x_n)} q(x_{n},u)< \ve,
$$
which verifies that the sequence $\{x_n\}$ is  forward Cauchy in the quasi-metric space $(X,q)$. The proof is complete.
\end{proof}

\begin{proposition} \label{E3E4} Assume that conditions {\rm(E1)} and {\rm(E2)} hold.   Assume also that the quasi-metric space $(X,q)$ is forward complete and the quasi-metric  enjoys the condition $q(x,y) = 0 \Longleftrightarrow x = y$. 
 Then, condition {\rm(E3$'$)} 
\begin{itemize}
\item[\bf (E3$'$)] there is a backward limit $\ox \in X$ of the sequence $\{x_n\}$ such that $\ox \in \Phi(x_n)$ for all $n\in \mathbb{N}$
\end{itemize}
implies both conditions {\rm(E3)} and {\rm(E4)} in Theorem~\ref{unified-version}.
\end{proposition}
\begin{proof} Obviously, {\rm (E3$'$)} $\Longrightarrow$ {\rm (E3)}. Therefore, it remains to prove the implication {\rm (E3$'$)} $\Longrightarrow$ {\rm (E4)}.
By Proposition~\ref{prop-Cauchy}, every generalized Picard sequence is forward Cauchy. 
By the assumed  forward completeness property of $X$, it is forward convergent   to some forward limit $\underline{x} \in X$. The fulfillness of {\rm (E3$'$)} ensures the existence
of some backward limit $\ox$ of $\{x_n\}$, i.e., $\disp \lim_{n\to\infty}q(\ox,x_n) = 0$ such that
$$
\ox\in \Phi(x_n)\; \mbox{ for all }\; n\in \mathbb N.
$$
Taking into account condition {\rm(E2)}, $\ox$ is also a forward limit, i.e., $q(x_n,\ox) \to 0$ as $n\to \mathbb{N}$.   Taking into account the triangular inequality to estimate the quasi-distance between $\ox$ and $\underline{x}$, we have $q(\ox,\underline{x}) \leq q(\ox,x_n) + q(x_n,\underline{x})$ for all $n\in \mathbb{N}$ and thus $q(\ox,\underline{x}) = 0$. The additional condition imposed on the quasi-metric implies $\ox = \underline{x}$. Therefore, condition {\rm (E4)} holds.
The proof is complete.  
\end{proof}

\begin{proposition} \label{Prop-E3} Assume that conditions {\rm(E1)-(E2)} hold and the quasi-metric space $(X,q)$ is   forward complete.
 Then, condition {\rm(E3$''$)} 
\begin{itemize}
\item[\bf (E3$''$)] there is $\ox \in X$ is a  forward limit of the sequence $\{x_n\}$ such that $\ox\in \Phi(x_n)$ for all $n\in \mathbb{N}$
\end{itemize}
is equivalent to condition {\rm(E3)} in Theorem~\ref{unified-version}.
\end{proposition}
\begin{proof} Obviously, {\rm(E3$''$)} $\Longrightarrow$ {\rm(E3)}. To justify the reverse implication it is sufficient to show that the imposed conditions {\rm(E1)-(E3)} implies that the element $\ox$ in {\rm(E3)} is, indeed, a forward limit of $\{x_n\}$. By Proposition~\ref{prop-Cauchy}, the sequence $\{x_n\}$ satisfying {\rm(E1)-(E2)} is forward Cauchy. Fix an element $\ox$ satisfying {\rm(E3)}.  Condition~{\rm(E2)} implies that $\disp\lim_{n\to\infty} q(x_n,\ox) = 0$ which clearly verifies that $\ox$ is a forward limit. The proof is complete.
\end{proof}


Next, we derive from Theorem~\ref{unified-version} an extension of Theorem~\ref{DHM-bms}; cf. \cite[Corollary~4.5]{bms15b} which can be used to further generalize the Ekeland  variational principle and its equivalents.  

\begin{theorem}{\bf (an `all sequences' version of Theorem~\ref{unified-version}).} \label{alt-version}
Let $(X,q)$ be a quasi-metric space, $\Phi: X\tto X$ be a 
 set-valued mapping. Assume that 
\begin{itemize}
\item[\bf(F1)] if $u\in \Phi(x)$, then $\Phi(u)\subset\Phi(x)$ for all $u,x\in X$;
\item[\bf(F2)]  for any generalized Picard sequence $\{x_n\}\subset X$, i.e., $x_{n+1} \in \Phi(x_n)$ for all $n\in \mathbb{N}$, if 
$$\disp\lim_{n\to \infty}\disp\sup_{x\in\Phi(x_n)} q(x_n, x) = 0,$$ 
then there exists some element  $x_\ast \in X$ such that $x_\ast \in \Phi(x_n)$  for all $n\in \mathbb{N}$;
\item[\bf(F3)] any forward Cauchy generalized Picard sequence $\{x_n\}\subset X$  
has at most one forward limit;
\item[\bf(F4)] for each generalized Picard sequence $\{x_n\}\subset X$, $\disp\lim_{n\to +\infty}q(x_n, x_{n+1}) =0.$
\end{itemize}
Then, $\Phi$ has a nonvariant point $\overline{x}$ in the sense that $\Phi(\ox) \subset \{\ox\}$. If, furthermore, $\Phi(\ox) \neq \emptyset$, then it is   an invariant point of $\Phi$, i.e., $\Phi(\ox) = \{\ox\}$.
\end{theorem}
\begin{proof} Without any loss of   generality, we may assume that $\Phi(x) \neq \emptyset$ for all $x\in X$; otherwise, the result is trivial; any element $\ox\in X$ such that $\Phi(\ox) = \emptyset$ is an invariant point of $\Phi$ with $\emptyset = \Phi(\ox) \subset \{\ox\}$. By Theorem~\ref{unified-version}, it is sufficent to show the existence of a generalized Picard sequence satisfying
\begin{eqnarray*}\label{cond100}
\disp\lim_{n\to \infty}\disp\sup_{x\in\Phi(x_n)} q(x_n, x) = 0.
\end{eqnarray*}
Such a sequence can be inductively constructed by starting with an arbitrary element $x_0$ and then following the iterative process:
\begin{eqnarray} \label{Picard1}
x_{n+1}\in\Phi(x_{n})\;\mbox{ with }\;  
q(x_{n},x_{n+1}) \ge \sup_{x\in \Phi(x_{n})} q(x_{n};x) - 2^{-n}\; \mbox{ for } \; n = 0,1,2,\ldots
\end{eqnarray}
It is clear that the sequence  $\{x_n\}$ is well defined and that the convergence condition {\rm(F4)} tells us that the quasi-distances $q(x_n,x_{n+1})$ tend to zero as $n\to\infty$. Taking into account the inequality in (\ref{Picard1}) ensures that $\disp\sup_{x\in\Phi(x_n)} q(x_n, x) \to 0$ as $n\to\infty$. The proof is complete.
\end{proof}
\begin{remark}

Theorem~\ref{alt-version} is an extension of \cite[Corollary~4.5]{bms15SVVA} (Theorem~\ref{DHM-bms}) due to Proposition 2.4 and the fact that forward (left-sequential) Hausdorff property implies the fulfillment of  \textbf{(F3)}.
The following example illustrates the usage of Theorem 2.1:

Let $X = [0,1]$ and the quasi-metric on $X$ defined by
\begin{eqnarray*}
q(x,y) = \begin{cases}
x - y & \mbox{if }\; x\geq y,\\
1&\mbox{if }\; x < y.
\\
\end{cases}
\end{eqnarray*}
 It is not dificult to check that this space is forward complete and forward Hausdorff. Given a forward sequence $\{x_n\}$ in $X$. For $n$ sufficiently large, we have $q(x_n,x_{n+1}) < 1$. The structure of $q$ yields $x_{n+1} < x_n$ and thus the sequence is decreasing eventually. Since it is bounded from below by $0$, it converges to a unique number in $X = [0,1]$.

Next, we will show that the space $(X,q)$ is not forward-backward complete. Consider the sequence $\{x_n\}$ where $x_n = n^{-1}$. Since $q(x_n, x_m) = n^{-1} - m^{-1} < n^{-1}$ for all $m,n\in \mathbb{N}$ with $m > n$, $\{x_n\}$ is obviously a forward Cauchy sequence. Since $q(0,x_n) = 1$ for all $n\in \mathbb{N}$, the sequence $\{x_n\}$  fails to backward converge to $0$.  Fix now an arbitrary number $\ox \in (0,1]$. We have $q(\ox,x_n ) = \ox - x_n$ for all $n$ sufficiently large ($n > 1/\bar{x}$) and thus $\disp \lim_{n\to \infty} q(\ox,x_n ) = \ox > 0$, i.e., $\ox$ is not a backward limit of $\{x_n\}$. Since  the chosen  forward Cauchy sequence $\{x_n\}$ has no  backward limit, the space is not forward-backward complete.  

Consider now a set-valued mapping $\Phi: X \rightrightarrows  X$  with images $\Phi(x) = [0,x]$. Obviously, conditions {\rm(E1)} and {\rm(E2)} are satisfied  by the chosen sequence with $\Phi(x_n) = [0, n^{-1}]$ and $\disp\sup_{x\in \Phi(x_n)} q(x_n, x) = n^{-1}$. Condition {\rm(E3)} is fulfilled for $\ox = 0$. We now show that $0$ is the only forward limit of $\{x_n\}$. Take an arbitrary number $\ox\in (0,1]$. For any interger $n\in \mathbb{N}$ with $n > 1/\ox$, one has $x_n = n^{-1} < \ox$ and thus $q(x_n, \ox) = 1$ clearly veryfying that $\ox$ is not a forward limit of $\{x_n\}$. Theorem~\ref{unified-version} ensures that $0$ is an invariant point of $\Phi$.
\end{remark}

Next, we will derive from Theorem~\ref{unified-version} Qiu's  revised preorder principle in \cite{q14b}.

\begin{theorem}  Theorem~\ref{unified-version}\; $\Longrightarrow$\; Theorem~\ref{DHM-qiu}.
\end{theorem}
\begin{proof} Assume that all the assumptions in Theorem~\ref{DHM-qiu} hold. We construct a bifunction $q: X\times X\to \mathbb{R}_+$  with
$$
q(x,y) := \vert\;\varphi(x) - \varphi(y)\;\vert \; \mbox{ for all }\; x,y\in X.
$$ 
Due to condition~{\rm(C1)} the function $\varphi$ is finite valued over $S(x_0)$. The pair $(S(x_0),q)$ is a quasi-metric space since $q(x,x) = \vert\;\varphi(x) - \varphi(x)\;\vert = 0$ for all $x\in X$ and 
\begin{eqnarray*}
q(x,z) & = & \vert\;\varphi(x) - \varphi(z)\;\vert = \vert\;(\varphi(x) - \varphi(z)) + (\varphi(z) - \varphi(y))\;\vert\\[1ex]
& \leq &  \vert\;\varphi(x) - \varphi(y)\;\vert + \vert\;\varphi(x) - \varphi(y)\;\vert\\[1ex]
& = & q(x,y) + q(y,z) \; \mbox{ for all }\; x,y,z \in X.
\end{eqnarray*}
In order to employ Theorem~\ref{common-point} we need to show that the level-set mapping satisfies all four conditions {\rm(E1)}--{\rm(E5)}.

First, let us construct a generalized Picard sequence starting with $x_0$ satisfying condition {\rm(E2)} as follow:
\begin{eqnarray*}\label{qiu-seq}
\begin{cases}
\mbox{If }\; S(x_{n-1}) = \emptyset, \; \mbox{then STOP};\\[1ex]
\mbox{If }\; S(x_{n-1}) \neq \emptyset, \mbox{then   choose }\; x_{n} \in S(x_{n-1}) \; \mbox{ with } \; \varphi(x_{n}) < \disp\inf_{u \in S(x_{n-1})} \varphi(u) + 2^{-n}.
\end{cases}
\end{eqnarray*}
If there exists $n$ such that $S(x_n) = \emptyset$, then we may take $\overline{x} = x_n$ and clearly
it satisfies {\rm(i)} and {\rm(ii)}. If not, we can obtain a sequence $\{x_n\}\subset S(x_0)$ with $x_{n+1} \in S(x_{n})$ for all $n\in \mathbb{N}$ such that 
$$
\varphi(x_n) < \disp\inf_{u \in S(x_{n-1})} \varphi(u) + 2^{-n}.
$$
Obviously, $\varphi(x_n) - \disp\inf_{u\in S(x_{n-1})} \varphi(u) \to 0$ as $n\to \infty$. 

Fix an arbitrary integer $n\in \mathbb{N}$. We get from $S(x_n) = \{x\in X| \; u\preceq x_n\}$ that for any $x\in S(x_n)$, $\varphi(x) \leq \varphi(x_n)\; \Longleftrightarrow \; \varphi(x_n) - \varphi(x) \leq 0$ due to the monotonicity of $\varphi$. Since
\begin{eqnarray*}
& & \varphi(x_{n}) - \inf_{x\in S(x_{n-1})} \;\varphi(x)\\[1ex]
& = & \varphi(x_{n}) + \sup_{x\in S(x_{n-1})}\big(-\varphi(x)\big)\; = \; \sup_{x\in S(x_{n-1})}\;\big(\varphi(x_{n}) -\varphi(x)\big)\\[1ex] 
& = & \sup_{x\in S(x_{n-1})}\;\big\vert\varphi(x_{n}) -\varphi(x)\big\vert = \sup_{x\in S(x_{n-1)}} q(x_{n};S(x_{n-1})),
\end{eqnarray*}
passsing to the  limit as $n\to\infty$, we derive   that
$$
\lim_{n\to \infty} \sup_{x\in S(x_{n-1)}} q(x_{n};S(x_{n-1})) = 0,
$$
which establishes that condition {\rm(E2)} holds for the sequence $\{x_n\}$.
Then, condition~{\rm(C3)} ensures the existence of an element $\overline{x}$ satisfying $\overline{x} \in S(x_n)$ for all $n\in \mathbb{N}$, i.e., condition~{\rm(E3)}.

By the transitivity of $\preceq$ and the structure of $S$, conditions {\rm(E1)}  and {\rm(E5)} hold. We also have $x\in S(x)$ for all $x\in S(x_0)$ due to the reflexivity of $\preceq$. 

It remains to check condition {\rm(E4)}. Assume now that $\disp\lim_{n\to\infty} q(x_n, \overline{x}) = \disp\lim_{n\to\infty} q(x_n, \underline{x}) = 0$ for two elements $\overline{x}$ and $\underline{x}$ in $X$. Fix $n\in \mathbb{N}$, we get from $\overline{x}, \underline{x}\in S(x_k)$ that $\varphi(\overline{x}) \leq \varphi(x_n)$ and $\varphi(\underline{x}) \leq \varphi(x_n)$ due to the definition of $S$ and the monotonicity of $\varphi$. Therefore, we have  
\begin{eqnarray*}
& & \disp\lim_{n\to\infty} q(x_n, \overline{x}) = \disp\lim_{n\to\infty} q(x_n, \underline{x}) = 0\\[1ex]
& \st{def.}{\Longleftrightarrow} & \disp\lim_{n\to\infty} \big\vert\;\varphi(x_n) - \varphi(\overline{x})\big| = \disp\lim_{n\to\infty} \big\vert \varphi(x_n) - \varphi(\underline{x})\big\vert = 0\\[1ex]
& {\Longleftrightarrow} & \disp\lim_{n\to\infty} \varphi(x_n) - \varphi(\overline{x}) = \disp\lim_{n\to\infty} \varphi(x_n) - \varphi(\underline{x}) = 0\\[1ex]
\end{eqnarray*}
which clearly implies that $\varphi(\overline{x})$ = $\varphi(\underline{x})$. By condition~{\rm(C2)}, $\overline{x}$ = $\underline{x}$ clearly verifying the fulfilment of condition {\rm (E4)}. Indeed, if $\overline{x} \neq \underline{x}$, then we get from {\rm(C2)} that $\varphi(\overline{x}) > \min\{\varphi(\overline{x}), \varphi(\underline{x})\}$ contradicting to $\varphi(\overline{x})\neq \varphi(\underline{x})$. \vspace*{.05in}

We have checked that the chosen sequence $\{x_k\}$ satisfies all the assumptions in Theorem~{\ref{unified-version}}. Employing this result to $\{x_k\}$, we obtain $S(\ox) = \{\ox\}$, i.e., $\ox$ is  an invariant point of $S$. The proof is complete. 
\end{proof}

Finally, let us derive also the Khank-Quy's result in \cite[Lemma~3.4]{kq10}. 

\begin{theorem}  Theorem~\ref{unified-version}\; $\Longrightarrow$\; Theorem~\ref{DHM-kq}. 
\end{theorem}
\begin{proof} Assume that there is a convergent sequence $\{x_n\}$ such that its limit $\overline{x}$ in the metric space $(X,d)$ satisfies  all three conditions {\rm(B1)--(B3)} in Theorem~\ref{DHM-kq}. 

Define a function $q : X\times X \to \mathbb{R}_+$ by
\begin{eqnarray}\label{def:oq}
q(x,y) := \begin{cases}
p(x,y) & \mbox{\rm if}\quad x\neq y,\\
0 & \mbox{\rm if}\quad x = y.
\end{cases}
\end{eqnarray}
It is clear that $q$ is a quasi-metric on $X$ since the positivity and triangle inequality properties hold  for $q$ defined in (\ref{def:oq}) due to the definition of $q$ and condition ($\tau1$) respectively. Observe that {\rm(B1)} $\Longleftrightarrow$ {\rm(E1)} and {\rm(B3)} $\Longrightarrow$ {\rm(E3)}. Observe also that {\rm(B2)} $\Longrightarrow$ {\rm(E2)} since $q(x,y) \leq p(x,y)$ for all $x,y\in X$ and thus
$$
\disp\lim_{n\to \infty} \sup_{u\in \Phi(x_n)}\;q(x_n,u) \leq \disp\lim_{n\to \infty} \sup_{u\in \Phi(x_n)}\;p(x_n,u) = 0.
$$ 
Observe finally that condition (\rm{E4}) holds as well. Conditions {\rm(B1)} and {\rm(B2)} ensure  that 
\begin{eqnarray}\label{cauchy-p}
\lim_{n\to \infty}\sup_{m>n} p(x_{n},x_m) = 0.
\end{eqnarray}
Define a sequence $\{y_n\}$ by $y_n = x_{n+1}$ for all $n\in\mathbb{N}$. Then, we have $\lim_{n\to\infty} p(x_n, y_n) = 0$. This together with (\ref{cauchy-p}) implies, by ($\tau{3}$), $\disp\lim_{n\to\infty} d(x_n, x_{n+1}) = 0$ clearly implying that $\{x_n\}$ is a Cauchy sequence in the metric space $(X,d)$.  By Proposition~\ref{prop-Cauchy}, the sequence $\{x_n\}$ is forward Cauchy 
in $(X,q)$. Assume now that it has a forward limit $x_\ast$ in $(X,q)$ in the sense that $\disp\lim_{n\to \infty} q(x_n,x_\ast) = 0$. Then, we have $\disp\lim_{n\to \infty} p(x_n,x_\ast) = 0$ as well. Arguing by contradiction, assume that $\disp\lim_{n\to \infty} p(x_n,x_\ast) =\gamma > 0$. Taking into account the definition of $q$, we get the existance of a subsequence $\{x_{n_k}\}$ of $\{x_n\}$ with $x_{n_k}\equiv x_\ast$. This and condition {\rm(B2)} lead to a contradiction:
$$ 
0 < \gamma \leq \lim_{n_k \to \infty} \sup_{u\in \Phi(x_{n_k})}\;p(x_{n_k},u)
= \disp\lim_{n\to \infty} \sup_{u\in \Phi(x_n)}\;p(x_n,u) = 0.
$$
Let $y_n = x_\ast$ for all $n\in\mathbb{N}$, we have $\disp\lim_{n\to\infty} p(x_n, y_n) = 0$. This together with \eqref{cauchy-p} yields $\disp\lim_{n\to\infty}d(x_n,x_\ast) = 0$ by $(\tau3)$. Since the sequence $\{x_n\}$ is Cauchy in the complete metric space $(X,d)$, the limit $x_\ast$ is unique. Therefore, condition {\rm(E4)} is satisfied for the quasi-metric $q$ defined in (\ref{def:oq}).  

 We have proved that all the assumptions of Theorem~\ref{unified-version} are satisfied in the quasi-metric $(X,q)$. Therefore, the conclusion of Theorem~\ref{DHM-kq} follows from that of Theorem~\ref{unified-version}. The proof is complete.
\end{proof}

\section{Conclusions}
The main result of this paper, Theorems~\ref{unified-version}, provides an extension of Dancs-Heged\"us-Medvegyev's fixed point theorem which not only unify several recent generalized versions of this theorem but also further extend them to the quasi-metric space setting. This feature allows us to obtain new
applications to Ekeland's variational principle and Caristi's fixed point theorem by using our alternative result (Theorem~\ref{alt-version}).   
Following this way, we plan to extend this 
research to  the setting of $\lambda$-spaces and to  the setting of complete cone metric spaces introduced by Lin et al. in \cite{lwa11}.

\section*{Funding}
Research of the second author  was partially  supported by
Ministerio de Economıa y Competitividad under grant
MTM2011-29064-C03(03) and by the Gaspard Monge Program for Optimization and Operations Research (PGMO).

\bibliographystyle{gOPT}
\bibliography{myrefsbao-thera.bib}

\begin{thebibliography}{10}
\providecommand{\url}[1]{\normalfont{#1}}
\providecommand{\urlprefix}{Available from: }

\bibitem{AR73}
Ambrosetti~A, Rabinowitz~PH. Dual variational methods in critical point theory
  and applications. J Func Anal. 1973;\hspace{0pt}14:349--381.

\bibitem{BP87}
Borwein~JM, Preiss~D. A smooth variational principle with applications to
  subdifferentiability and to differentiability of convex functions. Trans Amer
  Math Soc. 1987;\hspace{0pt}303(2):517--527.

\bibitem{c76}
Caristi~J. Fixed point theorems for mappings satisfying inwardness conditions.
  Trans Amer Math Soc. 1976;\hspace{0pt}215:241--251.

\bibitem{dhm83}
Dancs~S, Heged\"us~M, Medvegyev~P. A general ordering and fixed-point principle
  in complete\ metric space. Acta Sci Math (Szeged).
  1983;\hspace{0pt}46(1-4):381--388.

\bibitem{e79}
Ekeland~I. Nonconvex minimization problems. Bull Amer Math Soc (NS).
  1979;\hspace{0pt}1(3):443--474.

\bibitem{bms15SVVA}
Bao~TQ, Mordukhovich~BS, Soubeyran~A. Fixed points and {V}ariational principles
  with {A}pplications to capability theory of wellbeing via variational
  rationality. Set-Valued Var Anal. 2015;\hspace{0pt}23(2):375--398.

\bibitem{bms15b}
Bao~TQ, Mordukhovich~BS, Soubeyran~A. Minimal points, variational principles,
  and variable preferences in set optimization,. J Nonlinear Convex Anal.
  2015;\hspace{0pt}To appear.

\bibitem{bks15}
Bao~TQ, Khanh~PQ, Soubeyran~A. Variational principles with generalized
  distances and applications to behavioral sciences. 2015;\hspace{0pt}To
  appear.

\bibitem{h05}
Ha~TXD. Some variants of the {E}keland variational principle for a\ set-valued
  map. J Optim Theory Appl. 2005;\hspace{0pt}124(1):187--206.

\bibitem{kq10}
Khanh~PQ, Quy~DN. A generalized distance and enhanced ekeland's variational\
  principle for vector functions. Nonlinear Anal.
  2010;\hspace{0pt}73(7):2245--2259.

\bibitem{kq11}
Khanh~PQ, Quy~DN. On generalized {E}keland's variational principle and\
  equivalent formulations for set-valued mappings. J Global Optim.
  2011;\hspace{0pt}49(3):381--396.

\bibitem{q14a}
Qiu~JH. A pre-order principle and set-valued {E}keland variational principle. J
  Math Anal Appl. 2014;\hspace{0pt}419(2):904--937.

\bibitem{q14b}
Qiu~JH. A revised pre-order principle and set-valued {E}keland variational
  principles with generalized distances. 2014;\hspace{0pt}arXiv:1405.1522.

\bibitem{rsv82}
Reilly~IL, Subrahmanyam~PV, Vamanamurthy~MK. Cauchy sequences in quasi
  pseudometric spaces. Monatsh Math. 1982;\hspace{0pt}93(2):127--140.

\bibitem{bms15JOTA}
Bao~TQ, Mordukhovich~BS, Soubeyran~A. Variational analysis in psychological
  modeling. J Optim Theory Appl. 2015;\hspace{0pt}164(1):290--315.

\bibitem{bms15a}
Bao~TQ, Mordukhovich~BS, Soubeyran~A. Fixed points and variational principles
  with applications to capability theory of wellbeing via variational
  rationality. Set-Valued Var Anal. 2015;\hspace{0pt}23:375--398.

\bibitem{menucci}
Mennucci~ACG. On asymmetric distances. Anal Geom Metr Spaces.
  2013;\hspace{0pt}1:200--231.

\bibitem{c13}
Cobza{\c{s}}~S. Functional {A}nalysis in {A}symmetric {N}ormed {S}paces.
  Frontiers in Mathematics; Birkh\"auser/Springer Basel AG, Basel; 2013.

\bibitem{deza}
Deza~MM, Deza~E. Encyclopedia of {D}istances. 3rd ed. Springer, Heidelberg;
  2014.

\bibitem{mainik05}
Mainik~A, Mielke~A. Existence results for energetic models for rate-independent
  systems. Calc Var Partial Differential Equations.
  2005;\hspace{0pt}22(1):73--99.

\bibitem{haus}
Hausdorff~F. Grundzüge der mengenlehre. Leipzig, Veit; 1914.

\bibitem{wilson}
Wilson~WA. On quasi-metric spaces. Amer J Math.
  1931;\hspace{0pt}53(3):675--684.

\bibitem{frechet}
Fr\'echet~M. Sur quelques points de calcul fonctionnel. Rend Circolo Mat di
  Palermo. 1906;\hspace{0pt}22:1--74.

\bibitem{smith}
Smith~TE. Shortest-path distances: An axiomatic approach. Geogr Anal.
  1989;\hspace{0pt}21(1):1--31.

\bibitem{kelly}
Kelly~JC. Bitopological spaces. Proc London Math Soc.
  1963;\hspace{0pt}13(3):71--89.

\bibitem{reilly}
Reilly~IL, Subrahmanyam~PV, Vamanamurthy~MK. Cauchy sequences in
  quasipseudometric spaces. Monatsh Math. 1982;\hspace{0pt}93(2):127--140.

\bibitem{lwa11}
Lin~LJ, Wang~SY, Ansari~QH. Critical point theorems and {E}keland type
  variational principle with applications. Fixed Point Theory Appl.
  2011;\hspace{0pt}2011:Art. ID 914624, 1--21.

\end{thebibliography}

\end{document}